\documentclass[A4, 12pts]{article}
\usepackage{graphicx} 

\usepackage[margin=2cm]{geometry}
\usepackage{amsmath}
\usepackage{amsfonts}
\usepackage{amsthm}
\usepackage{mathtools}
\usepackage{amssymb}
\usepackage{enumitem}
\usepackage{subcaption}
\usepackage{indentfirst}
\usepackage{hyperref}

\newtheorem{theorem}{Theorem}[section]
\newtheorem{lemma}[theorem]{Lemma}

\newtheorem{corollary}[theorem]{Corollary}

\newtheorem{question}[theorem]{Question}
\theoremstyle{remark}

\newtheorem*{remark*}{Remark}

\newcommand{\cF}{\mathcal F}

\title{Piercing boxes and products of trees}
\author{Wei Rao~\thanks{Laboratory of Discrete and Combinatorial Optimization, Moscow Institute of Physics and Technology, Dolgoprudny, Moscow region, 141700, Russia. (Email: \texttt{raowei1998@gmail.com})} \and Taehyun Eom~\thanks{Department of Mathematical Sciences, Gwangju Institute of Science and Technology (GIST), Jeonnam-Gwangju Special Metropolitan City, Republic of Korea (Email:\texttt{{taeheom0,minkikim}@gist.ac.kr}). Supported by a grant(RS-2025-00555526) from the National Research Foundation of Korea(NRF), funded by the Ministry of Science and ICT(MSIT) of the Korean Government.} \and Minki Kim\thanks{Corresponding Author: \texttt{minkikim@gist.ac.kr}.}~\footnotemark[2]}
\date{\today}

\begin{document}

\maketitle

\begin{abstract}
    We give a new upper bound on the piercing number $\tau(\mathcal{F})$ of a finite family $\mathcal{F}$ of axis-parallel rectangles in the plane when every $p$ members contain an intersecting pair. Namely, we prove $\tau(\mathcal{F}) = O(p(\log\log p)(\log\log\log p))$, improving the factor $(\log\log p)^2$ in the best previously known bound to $(\log\log p)(\log\log\log p)$. For each $d>2$, when $\mathcal{F}$ is a finite family of axis-parallel boxes in $\mathbb{R}^d$, this yields $\tau(\mathcal{F}) = O(p(\log p)^{d-2}(\log\log p)(\log\log\log p))$.

    We also study an analogous problem replacing $\mathbb{R}^d=\mathbb{R}\times\cdots\times\mathbb{R}$ with the Cartesian product $T_1\square\cdots\square T_d$ of trees $T_1,\ldots,T_d$, and boxes with objects of the form $S_1\square\cdots\square S_d$ where $S_i$ is a subtree of $T_i$. For a finite family of such objects under the same intersection assumption as above, we establish an upper bound $\tau(\mathcal{F})=O(p(\log p)^{d-1})$ for each $d$. We also construct examples showing that this upper bound cannot be improved in order of magnitude.
\end{abstract}

\section{Introduction}\label{sec:intro}
\noindent Throughout this paper, all logarithms without an indicated base are natural. For a given ground set $\Omega$, let $\cF$ be a finite family of nonempty subsets of $\Omega$. A~\emph{piercing set} for $\cF$ is a set $S\subseteq\Omega$ that meets every member of $\cF$. The \emph{piercing number} $\tau(\cF)$ is the minimal size of a piercing set for $\cF$. A \emph{packing} of $\cF$ is a subfamily of $\cF$ whose members are pairwise disjoint, and the \emph{packing number} $\nu(\cF)$ is the maximal size of a packing in $\cF$.
We study the relation between the piercing number and the packing number of boxes and Cartesian products of trees.

An \emph{axis-parallel box}, or simply a \emph{box}, in $\mathbb R^d$ is a product of $d$ nonempty intervals; intervals of length zero are allowed. A \emph{rectangle} is a box in \(\mathbb{R}^2\). We consider only closed and bounded boxes. It is well-known that $\tau(\mathcal{F})=\nu(\mathcal{F})$ when $\mathcal{F}$ is a finite family of intervals on a line. When $d=2$, it was conjectured by Wegner~\cite{wegner1965kombinatorisch} that $\tau(\mathcal{F})\leq 2\nu(\mathcal{F})-1$. This conjecture was recently disproved: for every $\varepsilon>0$, there exists a nonempty finite family $\cF$ of rectangles satisfying $\tau(\cF) \geq \left(\frac{5}{2}-\varepsilon\right)\nu(\mathcal{F})$ (see  \cite{ajwani2026counterexample, bal202664}). The weaker conjecture that $\tau(\mathcal{F})$ is linear in $\nu(\mathcal{F})$ remains open, and the previously best known bound is $\tau(\mathcal{B}) = O_d(\nu(\mathcal{B})(\log\nu(\mathcal{B}))^{d-2}(\log\log\nu(\mathcal{B}))^2)$ when $d\geq 2$~\cite{tomon2023lower}.
We give the following improvement.
\begin{theorem}\label{thm:upper-bound-for-boxes}
    There is a constant $C>0$ such that for every integer $d\geq2$, the following holds: every nonempty finite family $\mathcal B$ of axis-parallel boxes in $\mathbb R^d$ satisfies
    \[
        \tau(\mathcal B) \leq C\nu(\mathcal B)
        \binom{\lfloor\log_2\nu(\mathcal{B})\rfloor + d - 2}{d - 2}
        \left(\log \log(e^e+ \nu(\mathcal B)) \right) \Big(1+ \log \log\log (e^e+\nu(\mathcal B))\Big) .
    \]
    In particular, for fixed $d$ and $\nu\to\infty$,
    \[
        \tau(\mathcal B)
        =O_d\!\left(\nu(\mathcal B)(\log \nu(\mathcal B))^{d-2}(\log\log\nu(\mathcal B))(\log\log\log\nu(\mathcal B))\right).
    \]
\end{theorem}

We also study the piercing patterns of more general set systems. A \emph{tree} is a connected graph with no cycles and a \emph{path} is a tree where all vertices are adjacent to at most $2$ other vertices. We consider only finite trees. A connected subgraph of a tree is called a \emph{subtree}. It is known implicitly in \cite{gyarfas1970helly} and stated explicitly in \cite{alon2002covering, berger2004kkm} that $\tau(\cF)=\nu(\cF)$ for every finite family $\cF$ of subtrees in a tree. Given a finite family $\mathcal{F}$ of intervals on a line, by putting vertices on the endpoints of the intervals and connecting every two consecutive vertices by an edge, one can regard $\mathcal{F}$ as a family of subtrees in a path. From this point of view, it is natural to consider an analogue of Theorem~\ref{thm:upper-bound-for-boxes} by replacing each axis by a tree. A concrete description can be given by graph cartesian products.

The \emph{Cartesian (graph) product} of graphs $G_1,\ldots,G_d$, denoted by $G_1 \square G_2\square \cdots \square G_d$, is a graph on $V(G_1)\times\cdots\times V(G_d)$ such that two vertices $(u_1,\ldots,u_d)$ and $(v_1,\ldots,v_d)$ are adjacent if and only if $u_i=v_i$ for all $i\neq j$ for some $j$ and $u_j$ and $v_j$ are adjacent in $G_j$.

Throughout this paper, we consider only finite trees. Given trees $T_1,\ldots,T_d$, a \emph{subtree-product} of these trees is the Cartesian product $S_1\square\cdots \square S_d \subset T_1\square \cdots \square T_d$ where $S_i$ is a subtree of $T_i$. In this notion, a finite family of boxes in $\mathbb{R}^d$ can be regarded as a family of subtree-products of paths. To see this, for each $i$-th coordinate, consider the projections of the boxes in the family. Each projection of a box gives us an interval. Take all distinct endpoints of such intervals, and join consecutive endpoints in increasing order to form a path, say $P_i$. Now each box can be represented as a subtree-product of $P_1,\ldots,P_d$, and it is obvious from the construction that two boxes intersect if and only if the corresponding subtree-products intersect. For general subtree-products, we obtain an asymptotically optimal bound $\tau(\cF) = O(\nu(\cF)(\log\nu(\cF))^{d-1})$.

\begin{theorem}\label{thm:upper-bound-for-subtrees}
    For every integer $d \geq 1$ and nonempty finite family $\mathcal{F}$ of subtree-products of $d$ trees, we have $\tau(\cF) \leq \nu(\cF)
    \binom{\lfloor\log_2\nu(\cF)\rfloor + d - 1}{d - 1}
    $.
\end{theorem}

\begin{theorem}\label{thm:lower-bound-for-subtrees}
    For every integer $d\geq1$, there is a constant $c_d>0$ such that the following holds: for every positive integer $h$, there is a finite family $\mathcal{F}$ of subtree-products of $d$ trees such that $\nu(\cF)=2^{h+1}$ and $\tau(\cF)\geq c_d 2^{h+1}(h+1)^{d-1}$.
\end{theorem}
Theorem~\ref{thm:upper-bound-for-subtrees} is analogous to the same upper bound for boxes given by K\'{a}rolyi~\cite{k1991point}, and the proof only requires a slight modification. However, Theorem~\ref{thm:lower-bound-for-subtrees} shows that this upper bound is asymptotically optimal for subtree-products unlike the case of boxes, while $\tau=\nu$ holds for both families of intervals on a line and families of subtrees in a tree. For boxes, the current best lower bound when $d \geq 3$ is $\tau = \Omega_d(\nu \left(\frac{\log\nu}{\log\log\nu}\right)^{d-2})$ by Tomon~\cite{tomon2023lower}.

\section{Preliminaries}
\subsection{Fractional piercing and fractional packing}
\noindent Given a finite family $\mathcal{F}$ of nonempty sets, a \emph{fractional piercing} for $\cF$ is a function $w:\Omega\to[0,\infty)$ such that for every member $A \in \cF$, $\sum_{a\in A}w(a)$ is at least $1$. The minimum value of \(\sum_{p\in\Omega}w(p)\) over all fractional piercings is the \emph{fractional piercing number}, denoted by $\tau^*(\cF)$. That is,
\[
 \tau^*(\cF)=\min\left\{
  \sum_{p\in\Omega}w(p): w:\Omega\to[0,\infty), \displaystyle\sum_{p\in F}w(p)\geq1\quad \forall F\in\cF\right\}.
\]
A \emph{fractional packing} of $\cF$ is a function \(x:\mathcal{F}\to[0,\infty)\) such that for every point \(p\in \Omega\), $\sum_{F\in\cF, F\ni p}x(F)$ is at most $1$. The maximum value of \(\sum_{F\in\cF}x(F)\) over all fractional packings is the \emph{fractional packing number}, denoted by $\nu^*(\cF)$. That is,
\[
 \nu^*(\cF)=\max\left\{
  \sum_{F\in\cF}x(F): x :\cF\to [0,\infty),\displaystyle\sum_{F\ni p}x(F)\leq1  \quad\forall p\in\Omega\right\}.
\]

Let $\mathcal{F}^{\cap} = \{\mathcal{H}\subseteq\mathcal{F}: \bigcap_{A\in\mathcal{H}}A\neq\varnothing\}$.
For each $\mathcal{H}\in\mathcal{F}^\cap\setminus\{\varnothing\}$, take an element $a_\mathcal{H}$ in~$\bigcap_{A\in\mathcal{H}}A$, and let $V= \{a_\mathcal{H} : \mathcal{H} \in \mathcal{F}^\cap \setminus\{\varnothing\} \}$. The family $\mathcal{F}'=\{A \cap V: A\in\mathcal{F}\}$ has the same intersection pattern as $\mathcal{F}$ since for every nonempty subfamily $\mathcal{H}$ of $\mathcal{F}$, 
\[\bigcap_{A\in\mathcal{H}}A\neq\varnothing \iff \bigcap_{A\in\mathcal{H}}A\cap V\neq\varnothing.\]
It immediately follows that restricting the ground set to the finite set $V$ preserves the piercing, packing, fractional piercing, and fractional packing numbers. Thus there exists an optimal fractional piercing with finite support, and by applying the finite linear programming duality, we have
\[
 1\leq\nu(\cF)\leq\nu^*(\cF)=\tau^*(\cF)\leq\tau(\cF)
\]
when $\mathcal{F}$ itself is nonempty. We define all four parameters above to be zero for the empty family.

Moreover, when $\mathcal{F}\neq\varnothing$, normalizing an optimal fractional piercing $w$ supported on $V$ gives a finitely supported probability measure $\mu$ such that 
\[
\mu(A)
=\sum_{v\in A\cap V}\frac{w(v)}{\tau^*(\cF)}
\geq\frac1{\tau^*(\mathcal F)}\qquad\forall A\in\cF.
\]

The proof of Theorem~\ref{thm:upper-bound-for-boxes} relies on the following statement, which was implicitly proved using Lemma~2 and Equation~(2) in Section~5 of \cite{correa2015independent}.
\begin{theorem}\label{thm:boxes-fractional-packing}
    There is a constant $A\geq1$ such that for every nonempty finite family $\mathcal R$ of rectangles, $\tau^*(\mathcal R)\leq A \nu(\mathcal R) \log\log(e^e +\nu(\mathcal R))$.
\end{theorem}

\subsection{Balanced separator and dimension-lifting lemmas}\label{prelim:generalupperbound}
The following lemmas are analogous to the proof arguments for the upper bound for boxes by K\'{a}rolyi~\cite{k1991point}. We first show the existence of a balanced separator in a tree with respect to any nonnegative function that is superadditive on disjoint sets.
\begin{lemma}\label{lem:trees-balanced-separator}
    Let $T$ be a tree on a finite vertex set $V$.
    Consider a function $f:2^V\to\mathbb{R}_{\geq0}$ satisfying
    \begin{align}\label{eq:superadditivity}f(A)+f(B) \leq f(A\cup B)\qquad\forall A,B\in 2^V\text{ s.t. }A\cap B = \varnothing.\end{align}
    Then there is a vertex $v$ in $T$ such that the vertex sets $C_1,\ldots,C_m$ of the components of $T-v$ satisfy
    \begin{align}\label{eq:balanced-separator}f(C_i) \leq \frac{f(V)}{2} \;\;\;\text{and}\;\;\;\sum_{i=1}^{m}f(C_i)\leq f(V).\end{align}
\end{lemma}
\begin{proof}
    For every edge $e$ of $T$, $T-e$ has exactly two components, say the vertex sets are $A_e$ and $B_e$ and assume $f(A_e)\leq f(B_e)$. By \eqref{eq:superadditivity}, we have $f(A_e) + f(B_e) \leq f(V)$ for all $A_e$ and $B_e$. It follows that $f(A_e)\leq f(V)/2$. We orient $e$ toward $B_e$.
    The resulting orientation of the tree $T$ should have a sink $v$. Otherwise, following outgoing edges would eventually produce a directed cycle. We claim that $v$ is the desired vertex.

    Let $C_1,\ldots,C_m$ be the vertex sets of the components of $T-v$.
    By the choice of the orientation and the assumption that $v$ is a sink, we have $f(C_i) \leq f(V)/2$ for every $i$. Moreover, by \eqref{eq:superadditivity}, we have
    \[\sum_{i=1}^m f(C_i) \leq f(\bigcup_{i=1}^mC_i) = f(V\setminus\{v\}) \leq f(V).\]
    The last inequality follows from the monotonicity implied by \eqref{eq:superadditivity}: $f(A) \leq f(A)+f(B\setminus A) \leq f(B)$ whenever $A\subset B$.
\end{proof}

    Let $d\geq2$ and consider a finite family $\cF=\{S_1^{(j)}\square\cdots\square S_d^{(j)}: 1\leq j \leq n\}$ of subtree-products in the Cartesian product $T_1\square\cdots\square T_d$ of trees, where $S_i^{(j)}$ is a subtree of $T_i$. 
    Suppose $V$ is the vertex set of $T_d$ and for each $W\subseteq V$, let $\cF(W) =\{S_1^{(j)}\square\cdots\square S_d^{(j)}: 1\leq j \leq n, V(S_d^{(j)}) \subseteq W\}$. Define $f:2^V\to\mathbb{R}_{\geq0}$ by $f(W) = \nu(\cF(W))$. Clearly, this function $f$ satisfies \eqref{eq:superadditivity}. By Lemma~\ref{lem:trees-balanced-separator}, there exists a $v$ of $T_d$ satisfying \eqref{eq:balanced-separator} with respect to $f$. 
    Let $C_1,\ldots,C_m$ be the vertex sets of the components of $T_d-v$.
    Then $\cF_v = \cF\setminus\left(\bigcup_{i=1}^m\cF(C_i)\right)$ consists of all subtree-products where $S_d^{(j)}$ contains $v$.
    Therefore, for $\cF' = \{S_1^{(j)}\square\cdots\square S_{d-1}^{(j)}: 1\leq j\leq n, v\in S_d^{(j)}\}$, we have $\tau(\cF_v) = \tau(\cF')$, and hence the following holds:
    \begin{align}\label{eq:tree-recurrence}
    \tau(\cF) \leq \tau(\cF') + \sum_{i=1}^m{\tau(\cF({C_i}))},\;\;\;\nu(\cF(C_i))\leq\frac{\nu(\cF)}{2},\;\;\;\sum_{i=1}^{m}\nu(\cF(C_i))\leq\nu(\cF).
    \end{align}
    This recurrence relation gives an upper bound on $\tau(\cF)$ in terms of $\nu(\cF)$ through the following lemma.    
\begin{lemma}\label{lem:dimension-lifting}
    Let $k\in\mathbb{N}$, and consider a function $f:\mathbb{N}\times\mathbb{N}\to\mathbb{R}_{\geq0}$ and a nondecreasing function $g:\mathbb{N}\to\mathbb{R}_{\geq0}$. Suppose that $f(p,k)\leq pg(p)$ for every $p\geq 1$. Suppose also that for every $p\geq 1$ and $d>k$, there exists positive integers $p_1,\ldots,p_m$ (not necessarily $m\neq0$) such that $p_i\leq p/2$ for each~$i$, $\sum_{i=1}^m p_i \leq p$, and
    \begin{align}\label{eq:recurrence}f(p,d) \leq f(p,d-1)+\sum_{i=1}^{m}f(p_i,d)\;\;\forall p\geq1, d>k.\end{align}
    Then, for every $p\geq1$ and $d\geq k$,
    \[f(p,d) \leq p g(p) \binom{\lfloor\log_2 p\rfloor +d-k}{d-k}.\]
\end{lemma}
\begin{proof}
    We use induction on $d$, and on $p$ for each fixed $d>k$. The case $d=k$ follows from the assumption.    
    
    Fix $d>k$ and assume the bound for $f(p,d')$ holds for all $k\leq d' < d$. If $p=1$, then there is no positive integer at most $p/2$. Therefore, the recurrence relation gives $f(p,d) \leq f(1,d-1) \leq g(1)$, where the second inequality follows from the induction hypothesis.
    Now suppose $p\geq2$, and assume the bound for $f(p',d)$ holds for all positive integers $p'<p$. Each $p_i$ with $p_i \leq p/2$ clearly satisfies $\lfloor\log_2 p_i \rfloor \leq \lfloor\log_2 p \rfloor -1$.
    Since we also have $\sum_{i=1}^m p_i \leq p$ and $g(p_i) \leq g(p)$, by applying the induction hypothesis to $f(p,d-1)$ and $f(p_i, d)$, we obtain
    \begin{align*}
        f(p,d) &\leq pg(p)\binom{\lfloor\log_2 p \rfloor+d-k-1}{d-k-1} + \sum_{i=1}^{m}p_ig(p_i)\binom{\lfloor\log_2 p_i \rfloor+d-k}{d-k} \\
        & \leq pg(p)\left(\binom{\lfloor\log_2 p \rfloor+d-k-1}{d-k-1} + \binom{\lfloor\log_2 p \rfloor+d-k-1}{d-k}\right) = pg(p)\binom{\lfloor\log_2 p \rfloor+d-k}{d-k},
    \end{align*}
    as desired.
\end{proof}

Theorem~\ref{thm:upper-bound-for-subtrees} can be proved directly from the above lemmas and observations.
\begin{proof}[Proof of Theorem~\ref{thm:upper-bound-for-subtrees}]
We first show that for every finite family $\cF$ of subtree-products in $d$ trees $T_1,\ldots,T_d$, we have $\tau(\cF)\leq\nu(\cF)^d$. When $d=1$, it is already known that $\tau = \nu$ for a family of subtrees in a tree. For $d>1$, let $\cF=\{S_1^{(j)}\square\cdots\square S_d^{(j)}: 1\leq j\leq n, S_i^{(j)}\text{ is a subtree in }T_i\;\;\forall 1\leq i\leq d\}$ and $\cF_i = \{S_i^{(j)}: 1 \leq j \leq n\}$ for each $i$. Obviously, we have $\nu(\cF_i) \leq \nu(\cF)$. Take a smallest piercing set $P_i$. Then we have $|P_i| = \tau(\cF_i) = \nu(\cF_i) \leq \nu(\cF)$. Then the Cartesian product $P_1\times\cdots\times P_d$ is a piercing set for $\cF$. This shows $\tau(\cF)\leq\nu(\cF)^d$.

Now for integers $p, d \geq1$, define
\[f(p,d) = \max\{\tau(\cF): \cF\text{ is a finite family of subtree-products in $d$ trees}, \nu(\cF)\leq p\},\]
where the maximum ranges over all choices of the trees. By the above argument, the possible piercing numbers of such families form a nonempty bounded set of integers, so $f(p,d)$ exists. Also, $f(p,1)=p$.
 
From the recurrence relation \eqref{eq:tree-recurrence}, one can see that the function $f(p,d)$ satisfies the assumptions in Lemma~\ref{lem:dimension-lifting} with $k=1$ and $g(p)=1$. The upper bound in Theorem~\ref{thm:upper-bound-for-subtrees} immediately follows.
\end{proof}

\subsection{Lemmas for rectangles}
\noindent
For a nonempty finite set \(X\subseteq\mathbb{R}^2\), we denote by \(\langle X\rangle \) the smallest axis-parallel rectangle containing~\(X\). We define $\langle \varnothing \rangle = \varnothing$. We shall use the following lemmas to obtain the upper bound for rectangles.

\begin{lemma}\label{lem:choice-of-four}
Suppose \(\mu\) is a probability measure on \(\mathbb{R}^2\) supported on a finite set \(P\subseteq\mathbb{R}^2\). Then, for any \(\theta\in(0,1)\), we have
\[
\sum_{\overset{p_1,\cdots,p_4\in P}{\mu(\langle p_1,p_2,p_3,p_4\rangle)\geq\theta}}\mu(\{p_1\})\mu(\{p_2\})\mu(\{p_3\})\mu(\{p_4\})\geq\left(\frac{1-\theta}{4}\right)^4.
\]
\end{lemma}
\begin{proof}
For a real number $r$ and $\star\in\{<,\leq,\geq,>\}$, we define $\mu(x \star r) = \mu(\{(x,y)\in\mathbb{R}^2:x \star r\})$ and $\mu(y \star r) = \mu(\{(x,y)\in\mathbb{R}^2:y \star r\})$.
Since $P$ is finite, we can choose $a,b,c,d\in\mathbb{R}$ such that 
\begin{align*}
\mu(x<a)\leq\frac{1-\theta}{4}\leq\mu(x\leq a), \quad \mu(x>b)\leq\frac{1-\theta}{4}\leq\mu(x\geq b)\\
\mu(y<c)\leq\frac{1-\theta}{4}\leq\mu(y\leq c), \quad
\mu(y>d)\leq\frac{1-\theta}{4}\leq\mu(y\geq d).
\end{align*}
For example, we can take \(a=\inf\{r:\mu(x\leq r)\geq\frac{1-\theta}{4}\}\).
Since \(\frac{1-\theta}{4}<\frac{1}{2}\), it is clear that \(a\leq b\) and \(c\leq d\). Moreover, we have
\begin{align*}
\mu([a,b]\times[c,d]) \geq 1-\mu(x<a)-\mu(x>b)-\mu(y<c)-\mu(y>d)\geq 1-4\times\frac{1-\theta}{4} = \theta.
\end{align*}
For any choice of points \(p_1\in\{(x,y)\in P:x\leq a\}\), \(p_2\in\{(x,y)\in P:x\geq b\}\), \(p_3\in\{(x,y)\in P:y\leq c\}\) and \(p_4\in\{(x,y)\in P:y\geq d\}\), we have
$\mu(\langle p_1,p_2, p_3,p_4\rangle)\geq\mu([a,b]\times[c,d])\geq\theta$, so it follows that 
\begin{align*}
\sum_{\overset{p_1,\cdots,p_4\in P}{\mu(\langle p_1,p_2,p_3,p_4\rangle)\geq\theta}}\mu(\{p_1\})\mu(\{p_2\})\mu(\{p_3\})\mu(\{p_4\})\geq \mu(x\leq a)\mu(x\geq b)\mu(y\leq c)\mu(y\geq d)\geq\left(\frac{1-\theta}{4}\right)^4,
\end{align*}
as desired.
\end{proof}

\begin{lemma}\label{lem:num-well-organized}
Let $m$ and $n$ be positive integers, and let $X$ be a multiset of points $p_1,\cdots,p_{mn}\in\mathbb{R}^2$ (not necessarily pairwise distinct). Then there are at most \(3^{mn}\) partitions of $X$ into exactly \(n\) parts \(P_1,\dots,P_n\) with $|P_i|=m$ where rectangles \(\langle P_1\rangle,\dots\langle P_n\rangle\) are pairwise disjoint.
\end{lemma}
\begin{proof}
We actually prove that there can be at most $\binom{mn-m}{n-1}2^{(m-1)n}<3^{mn}$ many such partitions. 
We may assume that \(x\)-coordinates and \(y\)-coordinates are distinct for those \(mn\) points, by applying small perturbation if needed. Now, if such a partition is given, we assign each point a value in $\{-1,0,+1\}$ as follows; for each $i$, $p\in P_i$ is assigned with $0$ if it is the leftmost point of $P_i$, i.e., its $x$-coordinate is smallest in $P_i$, \(+1\) if it is not the leftmost point of $P_i$ and its $y$-coordinate is greater than all points of $P_i$ with smaller $x$-coordinate, and \(-1\) otherwise. We claim that distinct partitions have distinct assignments.

Suppose not. Then there are two distinct partitions \(X=P_1\cup\dots\cup P_n\) and \(X=Q_1\cup\cdots\cup Q_n\) of $X$ having the same assignment. Since every part has exactly one point assigned with \(0\), we may assume, by permuting the indices if needed, that \(P_i\) and \(Q_i\) have the same leftmost point. We will show that $\bigcup_{i\neq j}{(P_i \cap Q_j)}=\varnothing$. If not, we can take the leftmost point $p\in\bigcup_{i\neq j}{(P_i \cap Q_j)}$. 

Suppose that $p$ lies in $P_i \cap Q_j$ for some $i\neq j$ and that \(p\) is assigned with \(+1\). This assignment particularly implies that both $P_i$ and $Q_j$ contain at least one point whose $x$-coordinates and $y$-coordinates are smaller than those of $p$. Let $q$ be such a point with the largest $y$-coordinate. Without loss of generality, we may assume $q \in P_i$. Note that, by the minimality assumption of $p$, it must be $q \in P_i \cap Q_i$. On the other hand, $Q_j$ also contains a point $p'$ whose $x$-coordinate is smaller than $p$, and by the maximality assumption of $q$, the $y$-coordinate of $p'$ is smaller than $q$. Now, since $|P_i|=|Q_i|$ while $p\in P_i\setminus Q_i$, there exists a point $q' \in Q_i \setminus P_i$. By the minimality assumption of $p$, it must be that the $x$-coordinate of $q'$ is greater than $p$. Now, since $\{p,p'\}\subset Q_j$ and $\{q,q'\}\subset Q_i$, it must be $\langle \{p,p'\} \rangle \cap \langle \{q,q'\} \rangle = \varnothing$ from the assumption that  $\langle Q_j\rangle \cap \langle Q_i\rangle = \varnothing$. However, this is not true because the $x$-coordinate of $p$ lies between the $x$-coordinate of $q$ and that of $q'$, and the $y$-coordinate of $q$ lies between the $y$-coordinate of $p$ and that of $p'$. This is a contradiction.

If \(p\) is assigned with \(-1\), then both $P_i$ and $Q_j$ contain at least one point whose $x$-coordinate is smaller than that of $p$ while $y$-coordinate is greater than that of $p$. Then a similar argument applies by taking such a point with smallest $y$-coordinate. The same contradiction would be derived, and this shows that $\bigcup_{i\neq j}{(P_i \cap Q_j)}$ contains no~point.

Now, each assignment determines the partition uniquely. We note that the number of choice of points with $0$ is at most $\binom{mn-m}{n-1}$, because the leftmost point of $X$ should be assigned with $0$, and the rightmost $m-1$ points of $X$ should not be assigned with $0$. For the remaining $mn-n=(m-1)n$ points, there are at most $2^{(m-1)n}$ choices to assign $+ 1$ and $-1$. Thus the total number of partitions satisfying the condition is bounded by $\binom{mn-m}{n-1}2^{(m-1)n}$.
\end{proof}

\section{Piercing boxes}
\noindent In this section we prove Theorem~\ref{thm:upper-bound-for-boxes}. By the same argument as in the proof of Theorem~\ref{thm:upper-bound-for-subtrees}, it is sufficient to prove it when $d=2$.
\begin{theorem}\label{thm:upper-bound-for-boxes-dim2}
    Every nonempty finite family $\mathcal R$ of axis-parallel rectangles in the plane satisfies
    \[
        \tau(\mathcal R) \leq 2A\nu(\mathcal R) \log\log (e^e+\nu(\mathcal R) ) \left[1+8\log \left( 12e\cdot24^{1/4}A \log\log(e^e+\nu(\mathcal R)) \right) \right],
    \]
    where $A$ is a constant from Theorem~\ref{thm:boxes-fractional-packing}.
\end{theorem}
The main idea is that we can find a small set of points such that the unpierced rectangles have a small packing number. This yields a decreasing sequence of families whose fractional piercing numbers are decreasing geometrically, which gives a piercing set of controlled size. The resulting estimate also yields weighted weak \(\varepsilon\)-nets.

\subsection{Reducing the packing number of the unpierced subfamily}
We first show the existence of a small set of points such that the unpierced rectangles have a small packing number.
\begin{lemma}\label{lem:boxes-packing-reduction}
Let $r\geq1$ and let \(\cF\) be a nonempty finite family of rectangles and \(\mu\) be a probability measure supported on a finite set \(P\subseteq\mathbb{R}^2\) such that \(\mu(R)\geq 1/r\) for every \(R\in\cF\). For an integer \(k\geq 0\) and \(\theta\in(0,1)\), there exists a set \(S\subseteq P\) such that 
\[|S|\leq k\;\;\;\text{and}\;\;\;\nu(\{R\in\cF:R\cap S=\varnothing \})\leq\left\lfloor\frac{3e24^{1/4}}{1-\theta}re^{\frac{-k\theta}{4r}}\right\rfloor.\]
\end{lemma}
\begin{proof}
If \(k\geq|P|\), the statement is obvious by taking \(S=P\). Assume \(k<|P|\).

Fix an integer \(n\geq 1\) to be chosen later for given \(k,\theta\) and \(r\). 
For a point set \(S\subseteq P\), let \(\Phi_S\) be the set of all sequences $p_\bullet$ of $4n$ points \(p_1,\cdots,p_{4n}\) in $P$ such that 
\begin{itemize}
    \item rectangles \(\langle\{p_1,p_2,p_3,p_4\}\rangle\), \dots, \(\langle\{p_{4n-3},p_{4n-2},p_{4n-1},p_{4n}\}\rangle\) are pairwise disjoint,
    \item $\mu(\langle p_{4i-3}, p_{4i-2}, p_{4i-1}, p_{4i}\rangle)\geq\theta/r$ for all \(1\leq i\leq n\), and
    \item \(\left(\bigcup_{i=1}^n\langle p_{4i-3}, p_{4i-2}, p_{4i-1}, p_{4i}\rangle\right)\cap S=\varnothing\).
\end{itemize}
For such a sequence $p_\bullet$, define \(\mathcal{R}(p_\bullet)=\bigcup_{i=1}^n\langle p_{4i-3}, p_{4i-2}, p_{4i-1}, p_{4i}\rangle\) and \(\phi(p_1,\cdots,p_{4n})=\prod_{i=1}^{4n}\mu(\{p_i\})\). 
We will show that there exists a chain \(\varnothing=S_0\subseteq S_1\subseteq\cdots \subseteq S_{k}\subseteq P\) such that 
\[|S_j|=j\;\;\;\forall{1\leq j\leq k}\;\;\;\text{and}\;\;\;\sum_{p_{\bullet}\in\Phi_{S_j}}\phi(p_{\bullet})\leq \frac{n!24^n3^{4n}}{(4n)!}e^{-jn\theta/r}.\]

Given \(S\subseteq P\), we have
\begin{align*}
\sum_{p\in P\setminus S}\mu(\{p\})\sum_{p_{\bullet}\in\Phi_S,p\in\mathcal{R}(p_\bullet)}\phi(p_{\bullet})
& = \sum_{p_\bullet\in\Phi_S}\sum_{p\in\mathcal{R}(p_\bullet)\cap P}\mu(\{p\})\phi(p_{\bullet})\\
&=\sum_{p_\bullet\in\Phi_S}\mu(\mathcal{R}(p_\bullet))\phi(p_{\bullet})\\
&=\sum_{p_{\bullet}\in\Phi_S}\phi(p_\bullet)\sum_{i=1}^n\mu(\langle\{p_{4i-3}, p_{4i-2}, p_{4i-1}, p_{4i}\}\rangle)\geq\frac{n\theta}{r}\sum_{p_{\bullet}\in\Phi_S}\phi(p_{\bullet}).
\end{align*}
Since \(\mu\) is a probability measure, there exists \(p\in P\setminus S\) such that \[
\sum_{p_{\bullet}\in\Phi_S, p\in\mathcal{R}(p_\bullet)}\phi(p_{\bullet})\geq\frac{n\theta}{r}\sum_{p_{\bullet}\in\Phi_S}\phi(p_{\bullet}).\]
For such \(p\), we have
\begin{align*}
\sum_{p_{\bullet}\in\Phi_{S\cup\{p\}}}\phi(p_{\bullet})
= \sum_{p_{\bullet}\in\Phi_S}\phi(p_{\bullet}) - \sum_{p_{\bullet}\in\Phi_S, p\in\mathcal{R}(p_\bullet)}\phi(p_{\bullet}) \leq\left(1-\frac{n\theta}{r}\right) \sum_{p_{\bullet}\in\Phi_S}\phi(p_{\bullet})\leq e^{-n\theta/r}\sum_{p_{\bullet}\in\Phi_S}\phi(p_{\bullet}).
\end{align*}
Applying this argument repeatedly, we construct a chain \(\varnothing=S_0\subseteq S_1\subseteq\cdots\subseteq S_{k}\subseteq P\) such that
\[
\sum_{p_{\bullet}\in\Phi_{S_k}}\phi(p_{\bullet})\leq e^{-n\theta/r}\sum_{p_{\bullet}\in\Phi_{S_{k-1}}}\phi(p_{\bullet})\leq e^{-2n\theta/r}\sum_{p_{\bullet}\in\Phi_{S_{k-2}}}\phi(p_{\bullet})\leq\cdots\leq e^{-kn\theta /r}\sum_{p_{\bullet}\in\Phi_\varnothing}\phi(p_{\bullet}).
\]
By the construction, it is obvious that $|S_k|=k$.
If we choose \(4n\) points in \(\mathbb{R}^2\) with repetition allowed, by Lemma~\ref{lem:num-well-organized}, there exist at most \(3^{4n}\) partitions, and each partition gives at most \(n!(4!)^n\) sequences in total. Hence, if we consider rearrangement of a sequence \(p_\bullet\in P^{4n}\) by a permutation \(\sigma\in S_{4n}\) defined as \(\sigma\cdot p_{\bullet}=(p_{\sigma(1)},\cdots,p_{\sigma(4n)})\), we have \(|\{\sigma\in S_{4n}:\sigma\cdot p_{\bullet}\in\Phi_{\varnothing}\}|\leq n!24^n3^{4n}\). By the definition of the \(\phi\), its value is invariant under the rearrangement. Hence,
\begin{align*}
(4n)!\sum_{p_\bullet\in \Phi_{\varnothing}}\phi(p_\bullet)
=\sum_{\sigma\in S_{4n}}\sum_{\sigma\cdot p_\bullet\in\Phi_\varnothing}\phi(\sigma\cdot p_\bullet) &=\sum_{p_\bullet\in P^{4n}}\sum_{\sigma\in S_{4n}, \sigma\cdot p_{\bullet}\in\Phi_{\varnothing}}\phi(p_\bullet)\\
&\leq\sum_{p_\bullet\in P^{4n}}n!24^n3^{4n}\prod_{i=1}^{4n}\mu(\{p_i\}) \\
&=n!24^n3^{4n}\left(\sum_{p\in P}\mu(\{p\})\right)^{4n}=n!24^n3^{4n},
\end{align*}
since \(\mu\) is a probability measure. Thus, \(\sum_{p_{\bullet}\in\Phi_{S_k}}\phi(p_\bullet)\leq\frac{n! 24^{n}3^{4n}}{(4n)!}e^{-kn\theta/r}\), as required.

Now, for a rectangle \(R\in\cF\), by applying the Lemma~\ref{lem:choice-of-four} to the conditional probability measure \(\mu(X|R)=\frac{\mu(X\cap R)}{\mu(R)}\), we have
\[
\sum_{\overset{p_1,\cdots,p_4\in P\cap R}{\mu(\langle p_1,p_2,p_3,p_4\rangle)\geq\theta\mu(R)}}\frac{\mu(\{p_1\})}{\mu(R)}\frac{\mu(\{p_2\})}{\mu(R)}\frac{\mu(\{p_3\})}{\mu(R)}\frac{\mu(\{p_4\})}{\mu(R)}\geq\left(\frac{1-\theta}{4}\right)^4.
\]

Now, given \(k,\theta,r\), let $n=\left\lfloor\frac{3e\cdot 24^{1/4}}{1-\theta}re^{-\frac{k\theta}{4r}}\right\rfloor + 1$, and take $S=S_k$.
If \(\nu(\{R\in\cF:R\cap S_k=\varnothing\})\geq n\), then there are pairwise disjoint rectangles \(R_1,\cdots,R_n\) in $\cF$. We can construct a sequence $p_\bullet$ in \(\Phi_{S_k}\) by taking \(p_{4i-3},p_{4i-2},p_{4i-1},p_{4i}\in P\cap R_i\) such that \(\mu(\langle p_{4i-3},p_{4i-2},p_{4i-1},p_{4i}\rangle)\geq\theta\mu(R_i)\geq\theta/r\) for all $i$. There are \(n!\) orderings of these rectangles, so we obtain
\begin{align*}
\sum_{p_\bullet\in\Phi_{S_k}}\phi(p_\bullet)
&\geq n!\prod_{i=1}^{n}\sum_{\overset{p_{4i-3},\dots,p_{4i}\in P\cap R_i}{\mu(\langle p_{4i-3},\dots,p_{4i}\rangle)\geq \theta\mu(R_i)}}\mu(\{p_{4i-3}\})\mu(\{p_{4i-2}\})\mu(\{p_{4i-1}\})\mu(\{p_{4i}\})\\
&\geq n!\left(\frac{1-\theta}{4}\right)^{4n}\prod_{i=1}^{n}\mu(R_i)^{4}\geq n!\left(\frac{1-\theta}{4}\right)^{4n}\frac{1}{r^{4n}}.
\end{align*}
Since we have \(\sum_{p_{\bullet}\in\Phi_{S_k}}\phi(p_\bullet)\leq\frac{n! 24^{n}3^{4n}}{(4n)!}e^{-kn\theta/r}\), the above inequality gives us \[\frac{n!24^n3^{4n}}{(4n)!}e^{-kn\theta/r}\geq n!\left(\frac{1-\theta}{4}\right)^{4n}\frac{1}{r^{4n}}.\]
This implies, by applying the inequality \((4n)!\geq(4n/e)^{4n}\), that \[\frac{3e\cdot 24^{1/4}}{1-\theta}re^{\frac{-k\theta}{4r}}\geq n = \left\lfloor \frac{3e\cdot 24^{1/4}}{1-\theta}re^{\frac{-k\theta}{4r}}\right\rfloor + 1,\]
which is a contradiction. Therefore, it must be 
\[
\nu(\{R\in \cF:R\cap S=\varnothing\})\leq n-1=\left\lfloor \frac{3e\cdot 24^{1/4}}{1-\theta}re^{\frac{-k\theta}{4r}}\right\rfloor,
\]
as desired.
\end{proof}
\begin{remark*}
    Lemma~\ref{lem:boxes-packing-reduction} was motivated by the weak $\varepsilon$-net problem. The idea is to take a small sample so that the unhit heavy rectangles have a small packing number, and then use the known piercing bound of Correa et al.~\cite{correa2015independent} to meet the remaining rectangles with few additional points. Together, these points and the sample form a weak $\varepsilon$-net. This explains the seemingly unusual assumption $\mu(R)\geq1/r$ in the lemma: setting $r=1/\varepsilon$, it is exactly the condition that $R$ is $\varepsilon$-heavy.
\end{remark*}

\subsection{The planar bound and weighted weak \texorpdfstring{$\varepsilon$}{epsilon}-nets}
\noindent
We then give an upper bound on the piercing number of rectangles in terms of the the fractional piercing number and the packing number.
\begin{lemma}\label{lem:boxes-fractional-transference}
Let \(\mathcal{F}\) be a finite nonempty family of rectangles and \(g:[1,\infty)\to[1,\infty)\) be a nondecreasing function satisfying \(\tau^*(\mathcal{G})\leq\nu(\mathcal{G})g(\nu(\mathcal{G}))\) for every nonempty subfamily \(\mathcal{G}\subseteq\cF\). Then, for any \(\theta,\rho\in(0,1)\), $\tau(\cF)$ is bounded from above by
\[
\tau(\cF)\leq\frac{\tau^*(\cF)}{1-\rho}\left[1 + \frac{4}{\theta}\log\left(\frac{3e\cdot24^{1/4}g(\nu(\cF))}{(1-\theta)\rho}\right)\right]
\]
In particular, for \(\rho=\theta=1/2\), we have $
\tau(\cF)\leq 2\tau^*(\cF)\left[1 + 8\log\left(12e\cdot24^{1/4}g(\nu(\cF))\right)\right]$.
\end{lemma}
\begin{proof}
We construct a decreasing sequence of families \(\cF=\cF_0\supsetneq\cF_1\supsetneq\cF_2\supsetneq\cdots\supsetneq\cF_{m}=\varnothing\) as follows.
Let \(c=\frac{4}{\theta}\log\left(\frac{3e\cdot24^{1/4}g(\nu(\cF))}{(1-\theta)\rho}\right)\).
If \(\cF_i\) is nonempty, then there exists a probability measure \(\mu_i\) with finite support such that \(\mu_i(R)\geq\frac{1}{\tau^*(\cF_i)}\) for every \(R\in\cF_i\). Applying Lemma~\ref{lem:boxes-packing-reduction} with \(r=\tau^*(\cF_i)\), \(k=\lceil c\tau^*(\cF_i )\rceil\) and given \(\theta\), there exists a point set \(S_i\) such that \(|S_i|\leq\lceil c\tau^*(\cF_i )\rceil\) and 
\begin{align*}
\nu(\{R\in\cF_i:R\cap S_i=\varnothing\})
&\leq\left\lfloor\frac{3e\cdot24^{1/4}}{1-\theta}\tau^*(\cF_i)e^{-\lceil c\tau^*(\cF_i)\rceil\theta/(4\tau^*(\cF_i))}\right\rfloor\\
&\leq\left\lfloor\frac{3e\cdot24^{1/4}}{1-\theta}\tau^*(\cF_i)e^{-c\theta/4}\right\rfloor\leq\frac{\rho}{g(\nu(\cF))}\tau^*(\cF_i).
\end{align*}
Define \(\cF_{i+1}=\{R\in\cF_i:R\cap S_i=\varnothing\}\).
Since \(\cF_{i+1}\subseteq\cF\) and \(g\) is nondecreasing, we have
\[
\tau^*(\cF_{i+1})\leq\nu(\cF_{i+1})g(\nu(\cF_{i+1}))\leq\frac{\rho\tau^*(\cF_i)}{g(\nu(\cF))}g(\nu(\cF_{i+1}))\leq\rho\tau^*(\cF_i)
\]
if \(\cF_{i+1}\not=\varnothing\), where \(\tau^*(\cF_{i+1})=0\leq\rho\tau^*(\cF_i)\) is still true for the case \(\cF_{i+1}=\varnothing\).
Hence, from the finiteness of \(\cF\), this process terminates after finitely many steps. Then, if \(\cF_m=\varnothing\), it means \(\cF_{m}=\{R\in\cF:R\cap(\bigcup_{i=0}^{m-1}S_i)=\varnothing\}\).
Thus \(\cup_{i=0}^{m-1}S_i\) is a piercing set of \(\cF\), and we have
\begin{align*}
\tau(\cF)&\leq\sum_{i=0}^{m-1}|S_i|\leq\sum_{i=0}^{m-1}\lceil c\tau^*(\cF_i)\rceil\leq\sum_{i=0}^{m-1}(c\tau^*(\cF_i) + 1)\\
&\leq\sum_{i=0}^{m-1}((c+1)\tau^*(\cF_i))\leq(c+1)\sum_{i=0}^{m-1}\rho^i\tau^*(\cF)\\
&\leq\frac{\tau^*(\cF)}{1-\rho}(1+c)=\frac{\tau^*(\cF)}{1-\rho}\left[1 + \frac{4}{\theta}\log\left(\frac{3e\cdot24^{1/4}g(\nu(\cF))}{(1-\theta)\rho}\right)\right]
\end{align*}
which completes the proof.
\end{proof}

Now we complete the proof of Theorem~\ref{thm:upper-bound-for-boxes-dim2}.
\begin{proof}[Proof of Theorem~\ref{thm:upper-bound-for-boxes-dim2}]
By Theorem~\ref{thm:boxes-fractional-packing}, we have $\tau^*(\mathcal{R})\leq A\nu(\mathcal{R})\log\log(e^e +\nu(\mathcal{R}))$ for every nonempty finite family \(\mathcal{R}\) of rectangles. Let \(g(t)=A\log\log(e^{e}+t)\). Then $g$ is a nondecreasing function on \([1,\infty)\) and \(g(1)>A\geq1\), so by applying Lemma~\ref{lem:boxes-fractional-transference}, we obtain
\begin{align*}
\tau(\mathcal R) 
&\leq 2\tau^*(\mathcal{R})\left[1+8\log \left( 12e\cdot24^{1/4}A \log\log(e^e+\nu(\mathcal R)) \right) \right]\\
&\leq2A\nu(\mathcal R) \log\log (e^e+\nu(\mathcal R) ) \left[1+8\log \left( 12e\cdot24^{1/4}A \log\log(e^e+\nu(\mathcal R)) \right) \right],
\end{align*}
as desired.
\end{proof}

For a probability measure $\mu$ on the plane and $0<\varepsilon\leq1$, a \emph{weighted weak $\varepsilon$-net} for rectangles is a finite set $N\subseteq\mathbb R^2$ meeting every rectangle $R$ with $\mu(R)\geq\varepsilon$. Its points need not belong to the support of $\mu$. As a corollary of the above proof, we obtain the following.
\begin{corollary}\label{cor:boxes-weighted-weak-net}
Let $\mu$ be a probability measure supported on a finite nonempty set $P\subseteq\mathbb R^2$. For every $r\geq1$, there is a weighted weak $(1/r)$-net $N$ for rectangles with
\[
 |N|\leq
 2r\left[1+8\log\left(12e\cdot24^{1/4}A\log\log(e^e+r)\right)\right],
\]
where \(A\) is the constant from Theorem~\ref{thm:boxes-fractional-packing}.
\end{corollary}
\begin{proof}
By disregarding $p\in P$ with $\mu(\{p\})=0$, we may assume that $\mu(\{p\})>0$ for all $p\in P$. For each rectangle $R$, we clearly have $\mu(\langle P \cap R\rangle) = \mu(R)$.
Note that when $P\cap R \neq \varnothing$, each side of the rectangle $\langle P \cap R\rangle$ contains at least one point in $P \cap R$. On the other hand, the rectangle $\langle P \cap R\rangle$ itself can be determined by exactly one point of $P$ on each side. Thus the number of possible rectangles obtained by $\langle P \cap R\rangle$ is at most $|P|^4$. In particular, \(\mathcal{R}=\{\langle([a,b]\times[c,d])\cap P\rangle:a,b,c,d\in\mathbb{R},a\leq b, c\leq d,\mu([a,b]\times[c,d])\geq 1/r\}\) is a finite family of rectangles. Also, by its definition, a piercing set of \(\mathcal{R}\) is a weighted weak \((1/r)\)-net for rectangles.

Since \(\mu(R)\geq1/r\) for every \(R\in\mathcal{R}\), it must be \(\nu(\mathcal{R})\leq r\). Also, for the function \(w:\mathbb{R}^2\to[0,\infty)\) defined as \(w(p)=r\mu(\{p\})\), we have \(\sum_{p\in R}w(p)=r\mu(R)\geq1\) for every \(R\in\mathcal{R}\). Therefore, \(w\) is a fractional piercing with $\sum_{p\in P}w(p) = r$, showing that \(\tau^*(\mathcal{R})\leq r\).
Now we apply the inequality \[\tau(\mathcal R) \leq 2\tau^*(\mathcal{R})\left[1+8\log \left( 12e\cdot24^{1/4}A \log\log(e^e+\nu(\mathcal R)) \right) \right]\] obtained in the proof of Theorem~\ref{thm:upper-bound-for-boxes-dim2}. This implies $\tau(\mathcal R)\leq 2r\left[1+8\log\left(12e\cdot24^{1/4}A\log\log(e^e+r)\right)\right]$.
\end{proof}

\section{Piercing subtree-products}
For every integer $d \geq 1$ and nonempty finite family $\mathcal{F}$ of subtree-products of $d$ trees, the upper bound $\tau(\cF) \leq \nu(\cF)\binom{\lfloor\log_2\nu(\cF)\rfloor + d - 1}{d - 1}$ in Theorem~\ref{thm:upper-bound-for-subtrees} was proved at the end of Section~\ref{prelim:generalupperbound}. It remains to prove Theorem~\ref{thm:lower-bound-for-subtrees}, which implies that the upper bound cannot be improved asymptotically.

For an integer $h\geq0$, let $T^h$ denote the rooted \emph{perfect binary tree of height $h$} whose vertices are all binary sequences of length at most \(h\). The root of $T^h$ is the empty sequence $\epsilon$, and each vertex $v$ of length at most $h-1$ has children $v0$ and $v1$. For example, consider a binary sequence $v=10$ of length $2$. This is a child of a binary sequence $1$ of length $1$, and has children $v0=100$ and $v1=101$. The binary sequences of length $h$ are precisely the leaves of $T^h$. We denote by $\Lambda_h$ the set of all leaves of $T^h$. For a vertex $u$, let $\ell(u)$ be the length of it, which equals its \emph{depth}, or the distance from the root. Let $T[u]$ be the subtree induced by $u$ and all its descendants. Note that \(|V(T[u])\cap \Lambda_h|=2^{h-\ell(u)}\).

\begingroup
\def\thetheorem{\ref{thm:lower-bound-for-subtrees}}
\begin{theorem}
For every integer $d\geq1$, there is a constant $c_d>0$ such that the following holds: for every positive integer $h$, there is a finite family $\mathcal{F}$ of subtree-products of $d$ trees such that $\nu(\cF)=2^{h+1}$ and $\tau(\cF)\geq c_d 2^{h+1}(h+1)^{d-1}$.
\end{theorem}
\addtocounter{theorem}{-1}
\endgroup
\begin{proof}
We first observe that $c_1 = 1$ by taking any tree \(T\) with exactly \(2^{h+1}\) vertices and the family \(\cF=\{\{v\}:v\in V(T)\}\). Clearly, we have \(\tau(\cF)=\nu(\cF)=2^{h+1}\).
Now let $d\geq2$.  Let $\mathbb{T}$ be the Cartesian product of $d$ copies of $T^h$, and let $\Lambda_h^d$ be the set of tuples of leaves in $\mathbb{T}$. Define a family
\[
    \cF_h=\left\{A_1\square\cdots\square A_d\subseteq\mathbb{T}:
    \begin{array}{l}
        A_i\text{ is a nonempty subtree of }T^h\ (i\in[d]),\\
        \displaystyle\prod_{i=1}^{d}|V(A_i)\cap\Lambda_h|
            \geq 2^{h(d-1)-1}
    \end{array}\right\}.
\]
Since every member of $\cF_h$ contains at least $2^{h(d-1)-1}$ points of $\Lambda_h^d$, we have $\nu(\cF_h)\leq\frac{2^{hd}}{2^{h(d-1)-1}}=2^{h+1}$. Conversely, the members of the form $\{a\}\square T[b]\square T^h\square\cdots\square T^h\subseteq\mathbb{T}$ with $a\in\Lambda_h$ and $\ell(b)=1$ are pairwise disjoint and each of such members contains exactly $2^{h-1+h(d-2)}=2^{h(d-1)-1}$ points of $\Lambda_h^d$. Since there are exactly $2^{h+1}$ such members, we have $\nu(\cF_h)=2^{h+1}$.

Consider a family $\mathcal{C}_h=
    \left\{T[u_1]\square\cdots\square T[u_d]\subseteq\mathbb{T}:
        u_i\in V(T^h),\ \sum_{i=1}^{d}\ell({u_i})=h\right\}$.
Clearly, there are $\binom{h+d-1}{d-1}$ many possibilities of $\ell(u_1),\ldots,\ell(u_d)$ such that $\sum_{i=1}^{d}\ell({u_i})=h$. For each of such possibilities, there are $2^{\sum_i\ell(u_i)}=2^h$ choices of the vertices $u_1,\ldots,u_d$ such that $\prod_{i=1}^{d}T[u_i]\in \mathcal{C}_h$. Moreover, each \(C\in\mathcal{C}_h\) has $\prod_{i=1}^d2^{h-\ell(u_i)}=2^{h(d-1)}$ many points of $\Lambda_h^d$. In summary, we have
\[
    |\mathcal C_h|=2^h\binom{h+d-1}{d-1},
    \text{ and } |V(C)\cap \Lambda_h^d|=2^{h(d-1)} \text{ for any } C\in\mathcal C_h.
\]

For $z=(x_1,\ldots,x_d)\in V(\mathbb{T})$ and $C=\prod_{i=1}^{d}T[u_i]$, define
\[
    w_C(z)=
    \begin{cases}
        \min\{2^{\ell(u_i)-\ell(x_i)}:1\leq i\leq d\},&z\in V(C),\\
        0,&z\notin V(C).
    \end{cases}
\]
If $\sum_{i=1}^d\ell(x_i)<h$, then \(\sum_{C\in\mathcal{C}_h}w_C(z)=0\).
Otherwise, let $s=\sum_{i=1}^d\ell(x_i)-h=\sum_{i=1}^d(\ell(x_i)-\ell(u_i))\geq0$. For every \(\prod_{i=1}^dT[u_i]\in\mathcal{C}_h\) containing~\(z\), either \(u_i=x_i\) or \(u_i\) is an ancestor of \(x_i\), so we have \(\ell(u_i)\leq \ell(x_i)\). Also, for every \(d\)-tuple \((a_1,\dots,a_d)\) of nonnegative integers with \(a_i\leq\ell(x_i)\) and \(\sum_{i=1}^da_i=s\), there exists unique vertex \(u_i'\) with \(\ell(u_i')=\ell(x_i)-a_i\) such that either \(u_i'=x_i\) or \(u_i'\) is an ancestor of \(x_i\). This implies that there is a one-to-one correspondence between \(C\in\mathcal{C}_h\) containing $z$ and \(d\)-tuple \((a_1,\dots,a_d)\) of nonnegative integers with \(a_i\leq\ell(x_i)\) and \(\sum_{i=1}^da_i=s\). Then we have
\begin{align*}
\sum_{C\in\mathcal{C}_h}w_C(z)&=\sum_{\overset{0\leq a_i\leq \ell(x_i),\;{i=1\dots d}}{a_1+\cdots+a_d = s}}\min\{2^{-a_i}:1\leq i\leq d\}\\&\leq \sup\{\sum_{\overset{a_1,\cdots,a_d\geq 0}{a_1+\cdots+a_d = s}}\min\{2^{-a_i}:1\leq i\leq d\} : s\geq 0\}=:K_d.
\end{align*}
Note that \(K_d\le \sup_{s\in\mathbb{Z}_{\geq0}} \binom{s+d-1}{d-1}2^{-s/d}\leq\sum_{t=0}^\infty\binom{t+d-1}{d-1}2^{-t/d}=(1-2^{-1/d})^{-d}<\infty\) for all $d\geq2$.

Let \(P\subseteq V(\mathbb{T})\) be a piercing set for \(\cF_h\). We claim that 
\[\sum_{z\in P}w_C(z)\geq 1/2\;\;\;\text{for all}\;\;\;C\in\mathcal{C}_h.\]
If this is true, then we have
\[
    \frac{1}{2}\times2^{h}\binom{h+d-1}{d-1}
    \leq\sum_{C\in\mathcal C_h}\sum_{z\in P}w_C(z)
    =\sum_{z\in P}\sum_{C\in\mathcal C_h}w_C(z)
    \leq K_d|P|.
\]
Since $\binom{h+d-1}{d-1} \geq \frac{(h+1)^{d-1}}{(d-1)!}$ and $\nu(\cF_h)=2^{h+1}$, we obtain
\[
\begin{aligned}
    \tau(\cF_h)
    \geq\frac{2^{h-1}}{K_d}\binom{h+d-1}{d-1}\geq\frac{2^{h-1}(h+1)^{d-1}}{K_d(d-1)!}=\frac{\nu(\cF_h)
        \bigl(\log_2\nu(\cF_h)\bigr)^{d-1}}
        {4K_d(d-1)!}.
\end{aligned}
\]
Thus the statement of the theorem holds with $c_d={1}/({4K_d(d-1)!)}>0$, which depends only on $d$.

It remains to prove the claim that \(\sum_{z\in P}w_C(z)\geq 1/2\) for all \(C\in\mathcal{C}_h\). Suppose not. Then there exists some $C=\prod_{i=1}^{d}T[u_i]\in\mathcal{C}_h$ such that $\sum_{z\in P}w_C(z)<1/2$. Since \(w_C(u_1,\cdots, u_d)=1\), it must be \((u_1,\cdots,u_d)\not\in P\). 
Let $P\cap V(C) = \{z_1,\ldots,z_m\}$ and 
\[\mathbb{S}_i=T[u_1]\square\cdots\square T[u_{j-1}]\square (T[u_j]-T[x_{i,j}])\square T[u_{j+1}]\square\cdots\square T[u_d],\] 
where $z_i = (x_{i,1},\ldots,x_{i,d})$ and $w_C(z_i) = 2^{\ell(u_j)-\ell(x_{i,j})}$. If there are several $j$'s satisfying $w_C(z_i) = 2^{\ell(u_j)-\ell(x_{i,j})}$, then just take one such $j$. Let $\mathbb{T}_0 = T[u_1]\square\cdots\square T[u_d]$, and for $1\leq i\leq m$, let $\mathbb{T}_{i} = \mathbb{T}_{i-1} \cap \mathbb{S}_i$. 
In this process, from the definition of $w_C$, since \((u_1,\cdots,u_d)\not\in P\), it never happens that $x_{i,j}=u_j$. This implies $(u_1,\ldots,u_d)\in \mathbb{T}_m$. In particular, we have $\mathbb{T}_m\neq\varnothing$. On the other hand, this process removes all members of $P$ contained in $C$, so we have $V(\mathbb{T}_m) \cap P = \varnothing$.
We will show that $|V(\mathbb{T}_m)\cap\Lambda_h^d | \geq 2^{h(d-1)-1}$, meaning that $\mathbb{T}_m \in \cF_h$. This leads a contradiction to the assumption that $P$ is a piercing set of~$\cF_h$.

From the equality $|V(\mathbb{T}_m)\cap\Lambda_h^d|=|V(C)\cap\Lambda_h^d|-\sum_{i=1}^{m}(|V(\mathbb{T}_{i-1})\cap\Lambda_h^d|-|V(\mathbb{T}_i)\cap\Lambda_h^d|)$ and the construction of $\mathbb{T}_i$'s, we have
\[1-\frac{|V(\mathbb{T}_m)\cap\Lambda_h^d|}{|V(C)\cap\Lambda_h^d|} = \sum_{i=1}^{m}\frac{|V(\mathbb{T}_{i-1})\cap\Lambda_h^d|-|V(\mathbb{T}_i)\cap\Lambda_h^d|}{|V(C)\cap\Lambda_h^d|} \leq  \sum_{i=1}^{m}\frac{|(V(C)\setminus V(\mathbb{S}_i))\cap\Lambda_h^d|}{|V(C)\cap\Lambda_h^d|}=\sum_{i=1}^m w_C(z_i)<\frac{1}{2},\]
where the last inequality is from the assumption. It follows that \[|V(\mathbb{T}_m)\cap\Lambda_h^d|>\frac{1}{2}|V(C)\cap\Lambda_h^d|=2^{h(d-1)-1}.\]
This completes the proof.
\end{proof}

\section{Remark}
\noindent
As mentioned in Section~\ref{sec:intro}, for rectangles, there is a gap between the upper bound $\tau=O(\nu(\log\log\nu)(\log\log\log\nu))$ and the linear lower bound. Another natural question arises for more general set systems, dealing with special type of convex sets. Here is one direction.

An \emph{abstract simplicial complex} $X$ with vertex set $V$ is a collection of subsets of $V$ that is closed under taking subsets: if $A \in X$ and $B \subset A$, then $B \in X$. Each member in $X$ is called a \emph{simplex}, or a \emph{face}, of $X$, and the dimension of $\sigma \in X$ is $\dim\sigma=|\sigma|-1$.
Given a finite family $\cF$ of nonempty sets, the \emph{nerve}, denoted by $N(\cF)$, is the complex on $\cF$ defined by \[N(\cF)=\{\mathcal{H} \subseteq \cF : \bigcap_{A\in\mathcal{H}}A \neq \varnothing\},\]
where the intersection of an empty family is regarded as the whole ground set. A complex that can be obtained as the nerve of a finite family of convex sets in $\mathbb{R}^d$ is said to be \emph{$d$-representable}. By the nerve theorem, every $d$-representable complex and its induced subcomplexes have trivial homology in dimension $d$ or greater. Here, the homology takes coefficient from a field. A complex on a finite vertex set with such homological property is called a \emph{$d$-Leray complex}. 

It is known that $1$-Leray complexes are precisely the complexes that can be obtained as the nerve of a finite family of subtrees in a tree. In this language, the nerve of boxes can be regarded as the intersection of $1$-representable complexes, and the nerve of subtree-products can be regarded as the intersection of $1$-Leray complexes. What we have studied in this paper can be translated as the relation between the maximal number of vertices containing no $1$-dimensional simplices and the minimal number of simplices whose union is equal to the vertex set of the complex. It is natural to ask what happens between the maximal number of vertices containing no $d$-dimensional simplices and the minimal number of simplices covering the vertex set in the intersection of $d$-representable/Leray complexes. 

We leave an explicit question for the intersection of two $2$-representable complexes, or the nerve of a finite family of sets of the form $A \times B \subset \mathbb{R}^2 \times \mathbb{R}^2$ where $A$ and $B$ are convex sets in the plane. For $p\geq3$, let $F(p)$ be the minimal positive integer such that every finite family of convex sets in the plane where every $p$ members contains a triple having a point in common can be pierced by $F(p)$ points. The current best upper bound on $F(p)$ is that $F(p)=O_\delta(p^{3+\delta})$ for every $\delta>0$, following from \cite{KellerSmorodinskyTardos2018IsraelJMath} and \cite{Rubin2022JACM}.
Analogously, for $p\geq3$, let $G(p)$ be the minimal positive integer such that every finite family of sets of the form $A \times B \subset \mathbb{R}^2 \times \mathbb{R}^2$ where $A$ and $B$ are convex sets in the plane where every $p$ members contains a triple having a point in common can be pierced by $G(p)$~points. An obvious bound is $F(p)\leq G(p)\leq (F(p))^2$. The first step toward understanding the intersection pattern of such objects in terms of piercing numbers could be the following:
\begin{question}
    Is $G(p) = o(F(p)^2)$ as $p\to\infty$?
\end{question}

\section*{Use of Generative-AI tools declaration}
\noindent
ChatGPT-6 Astra (OpenAI) was used to assist with proof development and manuscript preparation. The authors developed the initial strategy underlying the proof of Theorem~\ref{thm:upper-bound-for-boxes}. Subsequent discussions with ChatGPT contributed to the development and revision of the proofs, and the construction in Theorem~\ref{thm:lower-bound-for-subtrees} was initially suggested by ChatGPT. The authors have verified all arguments and take full responsibility for the final content.

\bibliographystyle{abbrv}
\bibliography{references}

\end{document}